\documentclass[leqno,11pt,a4paper]{amsart}
\usepackage[margin=2.7cm]{geometry}

\usepackage{pdfsync}
\usepackage[usenames,dvipsnames]{color}
\usepackage{hyperref}
\usepackage{mleftright}
\usepackage{amssymb}
\newtheorem{thm}{Theorem}[section]
\newtheorem*{theorem*}{Theorem}
\newtheorem{lemma}[thm]{Lemma}
\newtheorem{cor}[thm]{Corollary}
\newtheorem{prop}[thm]{Proposition}
\newtheorem*{proposition*}{Proposition}
\theoremstyle{definition}
\newtheorem{example}[thm]{Example}
\newtheorem{remark}[thm]{Remark}

\newtheorem{definition}[thm]{Definition}

\newtheorem{notation}[thm]{Notation}
\numberwithin{equation}{section}
\long\def\blankfootnotetext#1{\begingroup\def\thefootnote{\fnsymbol{footnote}}\footnotetext{#1}\endgroup}
\newcommand{\orig}{\mathbf{0}}
\newcommand{\Z}{\mathbb{Z}}

\newcommand{\Q}{\mathbb{Q}}

\newcommand{\C}{\mathbb{C}}
\newcommand{\A}{\mathbb{A}}
\newcommand{\Proj}{\mathbb{P}}
\newcommand{\abs}[1]{\left\vert{#1}\right\vert}
\newcommand{\Hom}[1]{\mathrm{Hom}\mleft({#1}\mright)}
\newcommand{\V}[1]{\mathrm{vert}\mleft({#1}\mright)}

\renewcommand{\gcd}[1]{\mathrm{gcd}\mleft\{{#1}\mright\}}
\newcommand{\mmax}{\mathrm{max}}
\newcommand{\mmin}{\mathrm{min}}
\newcommand{\scone}[1]{\mathrm{cone}\mleft\{{#1}\mright\}}
\newcommand{\sconv}[1]{\mathrm{conv}\mleft\{{#1}\mright\}}
\newcommand{\Sing}[1]{\mathrm{Sing}\mleft({#1}\mright)}
\newcommand{\dual}[1]{{#1}^\vee}
\newcommand{\mut}{\mathrm{mut}}
\newcommand{\NQ}{N_\Q}
\newcommand{\MQ}{M_\Q}
\newcommand{\modb}[1]{\left(\mathrm{mod}\ {#1}\right)}
\newcommand{\SC}[1]{\mathrm{SC}\mleft(#1\mright)}

\newcommand{\BB}{\mathcal{B}}
\newcommand{\res}[1]{\mathrm{res}\mleft({#1}\mright)}
\renewcommand{\emptyset}{\varnothing}
\begin{document}
\author[M.~E.~Akhtar]{Mohammad E.~Akhtar}
\author[A.~M.~Kasprzyk]{Alexander M.~Kasprzyk}
\address{Department of Mathematics\\Imperial College London\\London, SW$7$\ $2$AZ\\UK}
\email{mohammad.akhtar03@imperial.ac.uk}
\email{a.m.kasprzyk@imperial.ac.uk}
\blankfootnotetext{2010 \emph{Mathematics Subject Classification}: 14M25 (Primary); 14B05, 14J45 (Secondary).}
\title{Singularity content}
\maketitle
\begin{abstract}
 We show that a cyclic quotient surface singularity $\sigma$ can be decomposed, in a precise sense, into a number of elementary $T$-singularities together with a cyclic quotient surface singularity called the residue of $\sigma$. A normal surface $X$ with isolated cyclic quotient singularities $\{\sigma_i\}$ admits a $\Q$-Gorenstein partial smoothing to a surface with singularities given by the residues of the $\sigma_i$. We define the singularity content of a Fano lattice polygon $P$: this records the total number of elementary $T$-singularities and the residues of the corresponding toric Fano surface $X_P$. We express the degree of $X_P$ in terms of the singularity content of $P$; give a formula for the Hilbert series of $X_P$ in terms of singularity content; and show that singularity content is an invariant of $P$ under mutation.
\end{abstract}
\section{Introduction}
Let $C$ be a two-dimensional rational cone and let $X_C$ denote the corresponding affine toric surface singularity. Let $u$,~$v$ be primitive lattice points on the rays of $C$. Let $\ell$, the \emph{local index} of $C$, denote the lattice height of the line segment $uv$ above the origin and let $w$, the \emph{width} of $C$, denote the lattice length of $v-u$. Write $w = n\ell + \rho$ for $n,\rho\in\Z_{\geq 0}$ with $0 \leq\rho < \ell$. Then $X_C$ is a $T$-singularity~\cite{KS-B88} if and only if $\rho=0$, and we say that $X_C$ is an \emph{elementary $T$-singularity} if $n=1$ and $\rho=0$ (so $w=\ell$); these correspond to singularities of the form $\frac{1}{n\ell^2}(1,n\ell c-1)$ and $\frac{1}{\ell^2}(1,\ell c-1)$, respectively. Choose a decomposition of $C$ into a cone $R$, of width $\rho$ and local index $\ell$, and $n$ other cones, each of width and local index $\ell$. Then, up to lattice isomorphism, $R$ depends only on $C$ and not on the decomposition chosen (Proposition~\ref{prop:cone_decomposition}) and we give explicit formula for $R$ in terms of $C$. There is a $\Q$-Gorenstein deformation of $X_C$ such that the general fibre is the affine toric surface singularity $X_R$ (Proposition~\ref{prop:deformation_to_residual}). We call $X_R$ the residue of $C$, and write it as $\res{C}$. Given a normal surface $X$ with isolated cyclic quotient singularities $\{X_{C_i} : i \in I\}$ there exists a $\Q$-Gorenstein deformation of $X$ such that the general fibre is a surface with isolated singularities $\{\res{C_i} : i \in I\}$ (Corollary~\ref{cor:deformation_to_residual}).

Let $P$ be a Fano polygon and let $X_P$ denote the corresponding toric Fano surface defined by the spanning fan $\Sigma$ of $P$. For a cone $C_i$ of $\Sigma$ with width $w_i$ and local index $\ell_i$, write $w_i = n_i \ell_i + \rho_i$ with $0 \leq\rho_i < \ell_i$. The \emph{singularity content} of $P$ is the pair $(n,\BB)$ where $n = \sum_i n_i$ and $\BB$ is the cyclically-ordered list $\left\{\res{C_i}\right\}_i$ with empty residues omitted. We compute the degree of $X_P$ in terms of the singularity content of $P$ (Proposition~\ref{prop:degree_formula}) and express the Hilbert series of $X_P$, in the style of~\cite{Reid85}, as the sum of a leading term controlling the order of growth followed by contributions from the elements of $\BB$ (Corollary~\ref{cor:hilb_series}). The singularity content of $P$ is invariant under mutation~\cite{ACGK12}.

\section{Singularity content of a cone}\label{sec:singularity_content_cone}
Let $N$ be a lattice of rank two, and consider a (strictly convex) two-dimensional cone $C\subset\NQ:=N\otimes_{\Z}\Q$. Let $u$ and $v$ be the primitive lattice vectors in $N$ defined by the rays of $C$. Define the \emph{width} $w\in\Z_{>0}$ of $C$ to be the lattice length of $v-u$, and the \emph{local index} $\ell\in\Z_{>0}$ of $C$ to be the lattice height of the line segment $uv$ above the origin.

\begin{notation}\label{notation:crepant}
Given $C$, $u$, and $v$ as above, and a non-negative integer $m$ such that $m \leq w/\ell$, we define a sequence of lattice points $v_0,v_1,\ldots,v_{n+1}$ on $uv$ as follows:
\begin{enumerate}
\item\label{item:crepant1}
$v_0 = u$ and $v_{n+1} = v$;
\item\label{item:crepant2}
$v_{i+1} - v_i$ is a non-negative scalar multiple of $v-u$, for $i \in \{0,1,\ldots,n\}$;
\item\label{item:crepant3}
$v_{i+1} - v_i$ has lattice length $\ell$ for $i \in \{0,\ldots,\widehat{m},\ldots,n\}$;
\item\label{item:crepant4}
$v_{m+1} - v_m$ has lattice length $\rho$, with $0 \leq \rho < \ell$.
\end{enumerate}
\end{notation}

The sequence $v_0,\ldots,v_{n+1}$ is uniquely determined by $m$ and the choice of $u$. Note that $w = n \ell + \rho$. We consider the partition of $C$ into subcones $C_i:=\scone{v_i,v_{i+1}}$, $0 \leq i \leq n$.

\begin{lemma}[\protect{\cite[Proof of Proposition~3.9]{AK13}}]\label{lem:height_width_singularity}
If the cone $C \subset \NQ$ has singularity type $\frac{1}{r}(a,b)$ then $w = \gcd{r,a+b}$ and $\ell= r/\gcd{r,a+b}$.
\end{lemma}

\begin{prop}\label{prop:cone_decomposition}
Let $C\subset\NQ$ be a two-dimensional cone of singularity type $\frac{1}{r}(1,a-1)$. Let $u$,~$v$ be the primitive lattice vectors defined by the rays of $C$, ordered such that $u$,~$v$, and~$\frac{a-1}{r}u+\frac{1}{r}v$ generate $N$. Let $v_0,\ldots,v_{n+1}$ be as in Notation~\ref{notation:crepant}. Then:
\begin{enumerate}
\item\label{item:cone_decomposition1}
The lattice points $v_0,\ldots,v_{n+1}$ are primitive;
\item\label{item:cone_decomposition2}
The subcones $C_i$, $0 \leq i<m$, are of singularity type $\frac{1}{\ell^2}(1,\frac{\ell a}{w}-1)$;
\item\label{item:cone_decomposition3}
If $\rho\ne 0$ then the subcone $C_{m}$ is of singularity type $\frac{1}{\rho \ell}(1,\frac{\rho a}{w}-1)$;
\item\label{item:cone_decomposition4}
The subcones $C_i$, $m<i\leq n$, are of singularity type $\frac{1}{\ell^2}(1,\frac{\ell\bar{a}}{w}-1)$.
\end{enumerate}
Here $\bar{a}$ is an integer satisfying $(a-1)(\bar{a}-1) \equiv 1 \modb r$, and so exchanging the roles of $u$ and $v$ exchanges $a$ and $\bar{a}$. Note that the singularity type of $C_{m}$ depends only on $C$.
\end{prop}

\begin{proof}
Without loss of generality we may assume that $u = (0,1)$, that $v = (r,1-a)$, and that $m \ne 0$. The primitive vector in the direction $v-u$ is $(\alpha,\beta) := (\ell,{-a}/w)$. Thus $v_1 = (\alpha^2,1+\alpha \beta)$, and so $v_1$ is primitive. There exists a change of basis
sending $v_1$ to $(0,1)$ and leaving $(\alpha,\beta)$ unchanged. This change of basis sends $v_i$ to $v_{i-1}$ for each $1 \leq i \leq m$. It follows that the lattice points $v_i$, $1 \leq i \leq m$, are primitive, and that the cones $C_i$, $1 \leq i \leq m$, are
isomorphic. Since
\[
\textstyle
\frac{1}{\alpha^2}(\alpha^2,1+\alpha \beta) - \frac{1 + \alpha \beta}{\alpha^2} (0,1) = (1,0),
\]
we have that $C_1$ has singularity type $\frac{1}{\alpha^2}(1,{-1}-\alpha \beta) = \frac{1}{\ell^2}(1,\frac{\ell a}{w}-1)$. Since $r=w\ell$ (Lemma~\ref{lem:height_width_singularity}) we have that $\frac{\ell(a+kr)}{w}-1\equiv \frac{\ell a}{w}-1\modb{\ell^2}$ for any integer $k$, so that the singularity only depends on the equivalence class of $a$ modulo $r$. This proves~\eqref{item:cone_decomposition2}. Switching the roles of $u$ and $v$ proves~\eqref{item:cone_decomposition1} and~\eqref{item:cone_decomposition4}.

It remains to prove~\eqref{item:cone_decomposition3}. As before, we may assume that $u = (0,1)$ and $v = (r,1-a)$. Consider the change of basis described above. After applying this $m$ times, the cone $C_{m+1}$ has primitive generators $(0,1)$ and $(\rho \alpha, 1 + \rho \beta)$. Since
\[
\textstyle
\frac{1}{\rho \alpha}(\rho \alpha,1+\rho \beta) - \frac{1 + \rho \beta}{\rho \alpha} (0,1) = (1,0),
\]
we see that $C_{m+1}$ has singularity type $\frac{1}{\rho \alpha}(1,{-1}-\rho \beta) = \frac{1}{\rho \ell}(1,\frac{\rho a}{w}-1)$. Since $r/w=\ell$ (again by Lemma~\ref{lem:height_width_singularity}) we see that $\frac{\rho(a+kr)}{w}-1\equiv\frac{\rho a}{w}-1\modb{\rho\ell}$ for any integer $k$, hence the singularity only depends on $a$ modulo $r$.

Next, we need to show that~\eqref{item:cone_decomposition3} is well-defined: that is, that the  quotient singularities $\frac{1}{\rho \ell}(1,\frac{\rho a}{w}-1)$ and $\frac{1}{\rho \ell}(1,\frac{\rho\bar{a}}{w}-1)$ are equivalent. It is sufficient to show that
$$\left(\frac{\rho a}{w}-1\right)\left(\frac{\rho\bar{a}}{w}-1\right)\equiv 1\modb{\rho\ell}.$$
Let $k,c\in\Z_{\geq 0}$, $0\leq c<\rho\ell$ be such that
\begin{equation}\label{eq:cone_decomposition3_a}
\left(\frac{\rho a}{w}-1\right)\left(\frac{\rho\bar{a}}{w}-1\right)=k\rho\ell+c.
\end{equation}
From Lemma~\ref{lem:height_width_singularity} we see that $0\leq c<r$, and~\eqref{eq:cone_decomposition3_a} becomes
$$\left(a-1-\frac{nra}{d}\right)\left(\bar{a}-1-\frac{nr\bar{a}}{d}\right)=kr-\frac{knr^2}{d}+c,\quad\text{ where }d:=\gcd{r,a}\cdot\gcd{r,\bar{a}}.$$
Multiplying through by $d$ and reducing modulo $r$ we obtain $d(a-1)(\bar{a}-1)\equiv dc\modb{r}$. Suppose that $d\not\equiv 0\modb{r}$. Since $(a-1)(\bar{a}-1)\equiv 1\modb{r}$, we conclude that $c=1$.

Finally, suppose that $d\equiv 0\modb{r}$. Writing $a=a'\cdot\gcd{r,a}$ and $\bar{a}=\bar{a}'\cdot\gcd{r,\bar{a}}$, we obtain
$$1\equiv(a'\cdot\gcd{r,a}-1)(\bar{a}'\cdot\gcd{r,\bar{a}}-1)\equiv 1-a-\bar{a}\modb{r},$$
and hence $a\equiv-\bar{a}\modb{r}$. But this implies that $1\equiv(a-1)(-a-1)\equiv 1-a^2\modb{r}$ and so $a\mid r$. Hence $w=r$, $\ell=1$, and the singularity in~\eqref{item:cone_decomposition3} is equivalent to $\frac{1}{\rho}(1,\rho-1)$.
\end{proof}

Notice that the quantities $a/w$ and $\bar{a}/w$ appearing in Proposition~\ref{prop:cone_decomposition} are integers by Lemma~\ref{lem:height_width_singularity}.

\begin{definition}\label{def:sing_content_cone}
Let $C\subset\NQ$ be a cone of singularity type $\frac{1}{r}(1,a-1)$. Let $\ell$ and $w$ be as above, and write $w =n \ell + \rho$ with $0 \leq \rho < \ell$. The \emph{residue} of $C$ is given by
$$
\res{C}:=\left\{
\begin{array}{ll}
\frac{1}{\rho\ell}\big(1,\frac{\rho a}{w} - 1\big)&\text{ if }\rho\ne 0,\\
\emptyset&\text{ if }\rho=0.
\end{array}\right.
$$
The \emph{singularity content} of $C$ is the pair $\SC{C} := (n, \res{C})$.
\end{definition}

\begin{example}
Let $C$ be a cone corresponding to the singularity $\frac{1}{60}(1,23)$. Then $w=12$, $\ell=5$, and $\rho=2$. Setting $m=1$ we obtain a decomposition of $C$ into three subcones: $C_0$ of singularity type $\frac{1}{25}(1,9)$, $C_1$ of singularity type $\frac{1}{10}(1,3)$, and $C_2$ of singularity type $\frac{1}{25}(1,4)$. In particular, $\res{C}=\frac{1}{10}(1,3)$.
\end{example}

Recall that a $T$-singularity is a quotient surface singularity which admits a $\Q$-Gorenstein one-parameter smoothing; $T$-singularities correspond to cyclic quotient singularities of the form $\frac{1}{nd^2}(1,ndc-1)$, where $\gcd{d,c} = 1$~\cite[Proposition~3.10]{KS-B88}. We now show that $T$-singularities are precisely the cyclic quotient singularities with empty residue.

\begin{cor}\label{cor:tsing_decomposition}
Let $C\subset\NQ$ be a cone and let $w,\ell$ be as above. The following are equivalent:
\begin{enumerate}
\item\label{item:empty}
$\res{C} = \varnothing$;
\item\label{item:division}
There exists an integer $n$ such that $w = n\ell$;
\item\label{item:s^t}
There exists a crepant subdivision of $C$ into $n$ cones of singularity type $\frac{1}{\ell^2}(1,\ell c-1)$, $\gcd{\ell,c}=1$;
\item\label{item:tsing}
$C$ corresponds to a $T$-singularity of type $\frac{1}{n \ell^2}(1,n\ell c-1)$, $\gcd{\ell,c} = 1$.
\end{enumerate}
\end{cor}

\begin{proof}
\eqref{item:empty} and~\eqref{item:division} are equivalent by definition. \eqref{item:s^t} follows from~\eqref{item:division} by Proposition~\ref{prop:cone_decomposition}, and~\eqref{item:empty} follows from~\eqref{item:tsing} by Lemma~\ref{lem:height_width_singularity}. Assume~\eqref{item:s^t} and let the singularity type of $C$ be $\frac{1}{R}(1,A-1)$. The width of $C$ is $n$ times the width of a given subcone. Since $\gcd{\ell,c} = 1$, Lemma~\ref{lem:height_width_singularity} implies that
$$\gcd{R,A} = w = n\cdot\gcd{\ell^2,\ell c} = n\ell.$$
The local index of a given subcone coincides, by construction, with the local index of $C$. By Lemma~\ref{lem:height_width_singularity} we see that
$$R = \ell\cdot\gcd{R,A} = n\ell^2.$$
Finally, Proposition~\ref{prop:cone_decomposition} gives that $\ell A/w=\ell c$, hence $A=n\ell c$, and so~\eqref{item:s^t} implies~\eqref{item:tsing}.
\end{proof}

\subsection{Residue and deformation}\label{sec:residue_deformation}
Define the \emph{residue} of a cyclic quotient singularity $\sigma$ to be the residue of $C$, where $C$ is any cone of singularity type $\sigma$. The residue encodes information about $\Q$-Gorenstein deformations of $\sigma$.

\begin{prop}\label{prop:deformation_to_residual}
A cyclic quotient singularity $\sigma$ admits a $\Q$-Gorenstein smoothing if and only if $\res{\sigma}=\emptyset$. Otherwise there exists a $\Q$-Gorenstein deformation of $\sigma$ such that the general fibre is a cyclic quotient singularity of type $\res{\sigma}$.
\end{prop}

\begin{proof}
By definition, $\sigma$ admits a $\Q$-Gorenstein smoothing if and only if it is a $T$-singularity. Thus the first statement follows from Corollary~\ref{cor:tsing_decomposition}. Assume $\sigma$ is not a $T$-singularity and let $\omega$,~$\ell$, and~$\rho$ be as above. By Corollary~\ref{cor:tsing_decomposition} we must have $\rho>0$. Now $\sigma=\frac{1}{r}(1,a-1)$ has index $\ell$ and canonical cover
\[
\textstyle
\frac{1}{\omega}(1,-1)=(xy - z^{\omega})\subset\A^3_{x,y,z}.
\]
Taking the quotient by the cyclic group $\mu_{\ell}$, and noting that $\omega\equiv\rho\modb{\ell}$, we have:
\[
\textstyle
\frac{1}{r}(1,a-1) = (xy - z^\omega)\subset\frac{1}{\ell}(1,\frac{\rho a}{\omega} - 1,\frac{a}{w}).
\]
A $\Q$-Gorenstein deformation is given by
\[
\textstyle
(xy - z^{\omega} + tz^{\rho}) \subset \frac{1}{\ell}\left(1,\frac{\rho a}{\omega} - 1, \frac{a}{\omega}\right) \times \A^1_t,
\]
and the general fibre of this family is the cyclic quotient singularity $\frac{1}{\rho\ell}(1,\frac{\rho a}{\omega} - 1)$.
\end{proof}

By combining Proposition~\ref{prop:deformation_to_residual} above with the proof of Proposition~3.4 and the Remark immediately following it from~\cite{Tzi05}, which tells us that there are no local-to-global obstructions, we obtain:

\begin{cor}\label{cor:deformation_to_residual}
Let $H$ be a normal surface over $\C$ with isolated cyclic quotient singularities. There exists a global $\Q$-Gorenstein smoothing of $H$ to a surface $H^\mathrm{res}$ with isolated singularities such that $\Sing{H^\mathrm{res}}=\{\res{\sigma}\mid\sigma\in\Sing{H},~\res{\sigma}\neq\emptyset\}$.
\end{cor}

\section{Singularity content of a complete toric surface}\label{sec:content_toric_surface}
\begin{definition}\label{def:sing_content_fan}
Let $\Sigma$ be a complete fan in $\NQ$ with two-dimensional cones $C_1,\ldots,C_m$, numbered cyclically, with $\SC{C_i}=(n_i,\res{C_i})$. The \emph{singularity content} of the corresponding toric surface $X_\Sigma$ is
$$\SC{X_{\Sigma}}:=(n,\BB),$$
where $n:=\sum_{i=0}^m n_i$ and $\BB$ is the cyclically ordered list $\{\res{C_1},\ldots,\res{C_m}\}$, with the empty residues $\res{C_i}=\emptyset$ omitted. We call $\BB$ the \emph{residual basket} of $X_{\Sigma}$.
\end{definition}

\begin{notation}\label{notation:continued_fraction}
We recall some standard facts about toric surfaces; see for instance~\cite{Ful93}. Let $X$ be a toric surface with singularity $\frac{1}{r}(1,a-1)$. Let $[b_1,\ldots,b_k]$ denote the Hirzebuch--Jung continued fraction expansion of $r/(a-1)$, having length $k\in\Z_{>0}$. For $i\in\{1,\ldots,k\}$, define $\alpha_i,\beta_i\in\Z_{>0}$ as follows: Set $\alpha_1=\beta_k=1$ and set
\begin{eqnarray*}
\alpha_i/\alpha_{i-1}&:=&[b_{i-1},\ldots,b_1],\quad 2 \leq i\leq k,\\
\beta_i/\beta_{i+1}&:=&[b_{i+1},\ldots,b_k],\quad 1\leq i\leq k-1.
\end{eqnarray*}
If $\pi:\widetilde{X}\rightarrow X$ is a minimal resolution then
$$K_{\widetilde{X}} = \pi^*K_X + \sum_{i=1}^k d_iE_i,$$
where $E_i^2=-b_i$ and $d_i=-1+(\alpha_i+\beta_i)/r$ is the discrepancy.
\end{notation}

\begin{prop}\label{prop:degree_formula}
Let $X$ be a complete toric surface with singularity content $(n,\BB)$. Then
$$
K_X^2=12-n-\sum_{\sigma\in\BB}A(\sigma),\quad\text{ where }A(\sigma):=k_\sigma+1-\sum_{i=1}^{k_\sigma}d_i^2b_i+2\sum_{i=1}^{k_{\sigma}-1}d_i d_{i+1}.
$$
\end{prop}

\begin{proof}
Let $\Sigma$ be the fan in $\NQ$ of $X$. If $C\in\Sigma$ is a two-dimensional cone whose rays are generated by the primitive lattice vectors~$u$ and~$v$ then, possibly by adding an extra ray through a primitive lattice vector on the line segment $uv$, we can partition $C$ as $C=S\cup R_C$, where $S$ is a (possibly smooth) $T$-singularity or $S=\emptyset$, and $R_C=\res{C}$. Repeating this construction for all two-dimensional cones of $\Sigma$ gives a new fan $\widetilde{\Sigma}$ in $\NQ$. If $\widetilde{X}$ is the toric variety corresponding to $\widetilde{\Sigma}$ then the natural morphism $\widetilde{X} \to X$ is crepant. In particular $K_{\widetilde{X}}^2 = K_X^2$. Notice that $\SC{X} = (n, \BB) = \mathrm{SC}(\widetilde{X})$.

By resolving singularities on all the nonempty cones $R_{C}$, we obtain a morphism $Y \to \widetilde{X}$ where the toric surface $Y$ (whose fan we denote $\Sigma_{Y}$) has only $T$-singularities. Thus by Noether's formula~\cite[Proposition 2.6]{HP10}
\begin{equation}\label{eq:noether_formula}
K_Y^2+\rho_Y+\sum_{\sigma\in\Sing{Y}}{\mu_\sigma}=10,
\end{equation}
where $\rho_Y$ is the Picard rank of $Y$, and $\mu_\sigma$ denotes the Milnor number of $\sigma$. But $\rho_Y+2$ is equal to the number of two-dimensional cones in $\Sigma_Y$, and the Milnor number of a $T$-singularity $\frac{1}{nd^2}(1,ndc-1)$ equals $n-1$, hence
\begin{equation}\label{eq:noether_combinatorics}
\rho_Y+\sum_{\sigma\in\Sing{Y}}{\mu_\sigma}=-2+n+\sum_{\sigma\in\BB}(k_{\sigma} +1),
\end{equation}
where $k_{\sigma}$ denotes the length of the Hirzebuch--Jung continued fraction expansion $[b_1,\ldots,b_{k_{\sigma}}]$ of $\sigma\in\BB$. With notation as in Notation~\ref{notation:continued_fraction},
\begin{equation}\label{eq:degree_combinatorics}
K_Y^2 = K_{X}^2 + \sum_{\sigma \in \BB}\left(-\sum_{i = 1}^{k_{\sigma}}d_i^2b_i + 2\sum_{i=1}^{k_{\sigma}-1}d_id_{i+1}\right).
\end{equation}
Substituting~\eqref{eq:noether_combinatorics} and~\eqref{eq:degree_combinatorics} into~\eqref{eq:noether_formula} gives the desired formula.
\end{proof}

\begin{remark}
If $X$ has only $T$-singularities, or equivalently if $\BB=\emptyset$, then Proposition~\ref{prop:degree_formula} gives $K_X^2=12-n$.
\end{remark}

The $m$-th Dedekind sum, $m\in\Z_{\geq 0}$, of the cyclic quotient singularity $\frac{1}{r}(a,b)$ is
\[
\delta_m:=\frac{1}{r}\sum\frac{\varepsilon^m}{(1-\varepsilon^a)(1-\varepsilon^b)},
\]
where the summation is taken over those $\varepsilon\in\mu_r$ satisfying $\varepsilon^a\neq 1$ and $\varepsilon^b\neq 1$. By Proposition~\ref{prop:degree_formula} and~\cite[\S8]{Reid85} we obtain an expression for the Hilbert series of $X$ in terms of its singularity content:

\begin{cor}\label{cor:hilb_series}
Let $X$ be a complete toric surface with singularity content $(n,\BB)$. Then the Hilbert series of $X$ admits a decomposition
$$
\mathrm{Hilb}(X,-K_X)=\frac{1 + (K_X^2 - 2)t + t^2}{(1-t)^3} + \sum_{\sigma \in \BB}Q_{\sigma}(t), \quad\text{where } Q_{\frac{1}{r}(a,b)}:=\frac{\sum_{i = 0}^{r-1}(\delta_{(a+b)i} - \delta_0)t^i}{1-t^r}.
$$
\end{cor}

\subsection{Singularity content and mutation}\label{subsec:content_and_mutation}
A lattice polygon in $\NQ$ is called \emph{Fano} if $\orig$ lies its strict interior, and all its vertices are primitive; see~\cite{KN12} for an overview. The \emph{singularity content} of a Fano polygon $P$ is $\SC{P} := \SC{X_{\Sigma}}$, where $\Sigma$ is the spanning fan of $P$; that is, $\Sigma$ is the complete fan in $\NQ$ with cones spanned by the faces of $P$.

Under certain conditions, one can construct a Fano polygon $Q:=\mut_h(P,F)\subset\NQ$ called a \emph{(combinatorial) mutation} of $P$. Here $h\in M:=\Hom{N,\Z}$ is a primitive vector in the dual lattice, and $F\subset\NQ$ is a point or line segment satisfying $h(F)=0$. For the details of this construction see~\cite{ACGK12}.

\begin{prop}\label{prop:sing_content_is_invariant}
Let $Q:=\mathrm{mut}_h(P,F)$. Then $\SC{P} = \SC{Q}$. In particular, singularity content is an invariant of Fano polygons under mutation.
\end{prop}

\begin{proof}The dual polygon $\dual{P}\subset\MQ$ is an intersection of cones
$$\dual{P} = \bigcap\left(\dual{C}_L - v_L\right),$$
where the intersection ranges over all facets $L$ of $P$. Here $C_L\subset\NQ$ is the cone over the facet $L$ and $v_L$ is the vertex of $\dual{P}$ corresponding to $L$.

If $F$ is a point then $P\cong Q$ and we are done. Let $F$ be a line segment and let $P_\mmax$ and $P_\mmin$ (resp.~$Q_\mmax$ and $Q_\mmin$) denote the faces of $P$ (resp.~$Q$) at maximum and minimum height with respect to $h$. By assumption the mutation $Q$ exists, hence $P_\mmin$ must be a facet, and so there exists a corresponding vertex $v_0\in M$ of $\dual{P}$. $P_\mmax$ can be either facet or a vertex. The argument is similar in either case, so we will assume that $P_\mmax$ is a facet with corresponding vertex $v_1\in M$ of $\dual{P}$.

The inner normal fan of $F$, denoted $\Sigma$, defines a decomposition of $\MQ$ into half-spaces $\Sigma^+$ and $\Sigma^-$. The vertices $v_0$ and $v_1$ of $\dual{P}$ lie on the rays of $\Sigma$; any other vertex lies in exactly one of $\Sigma^+$ or $\Sigma^-$. Mutation acts as an automorphism in both half-spaces. Thus the contribution to $\SC{Q}$ from cones over all facets excluding $Q_\mmax$ and $Q_\mmin$ is equal to the contribution to $\SC{P}$ from cones over all facets excluding $P_\mmax$ and $P_\mmin$. Finally, mutation acts by exchanging $T$-singular subcones between the facets $P_\mmax$ and $P_\mmin$, leaving the residue unchanged. Hence the contribution to $\SC{Q}$ from $Q_\mmax$ and $Q_\mmin$ is equal to the contribution to $\SC{P}$ from $P_\mmax$ and $P_\mmin$.
\end{proof}

\begin{example}
If two Fano polygons are related by a sequence of mutations then the corresponding toric surfaces have the same anti-canonical degree~\cite[Proposition 4]{ACGK12}. The Fano polygons $P_1:=\sconv{(0,1),(5,4),(-7,-8)}$ and $P_2:=\sconv{(0,1),(3,1),(-112,-79)}$ correspond to $\Proj(5,7,12)$ and $\Proj(3,112,125)$, respectively. These both have degree $48/35$, however their singularity contents differ:
\[
\textstyle
\SC{P_1}=\left(12,\left\{\frac{1}{5}(1,1),\frac{1}{7}(1,1)\right\}\right),\qquad
\SC{P_2}=\left(5,\left\{\frac{1}{14}(1,9),\frac{1}{125}(1,79)\right\}\right).
\]
Hence they are not related by a sequence of mutations.
\end{example}

\begin{lemma}\label{lem:sing_number_bounds_picard_rank}
Let $P$ be a Fano polygon with $\SC{P}=(n,\BB)$, and let $\rho_X$ denote the Picard rank of the corresponding toric surface. Then $\rho_X\leq n+\abs{\BB}-2$.
\end{lemma}

\begin{proof}
The cone over any facet of $P$ admits a subdivision (in the sense of Notation~\ref{notation:crepant}) into at least one subcone. Therefore we must have that $\abs{\V{P}}\leq n+\abs{\BB}$. Recalling that $\rho_X=\abs{\V{P}}-2$ we obtain the result.
\end{proof}

Since singularity content is preserved under mutation, Lemma~\ref{lem:sing_number_bounds_picard_rank}  gives an upper bound on the rank of the resulting toric varieties.

\begin{example}
In~\cite{AK13} we classified one-step mutations of (fake) weighted projective planes. It is natural to ask how much of the graph of mutations of a given (fake) weighted projective plane is captured by the graph of one-step mutations. Lemma~\ref{lem:sing_number_bounds_picard_rank} shows that the two graphs coincide if the singularity content of the (fake) weighted projective plane in question satisfies $n+\abs{\BB}=3$. For example the full mutation graph of $\Proj^2$ is isomorphic to the graph of solutions of the Markov equation $3xyz = x^2 + y^2 + z^2$~\cite[Example 3.14]{AK13}. More interestingly, the weighted projective plane $\Proj(3,5,11)$ does not admit \emph{any} mutations~\cite[Example 3.5]{AK13}. 
\end{example}

\subsection*{Acknowledgements}
We thank Tom Coates, Alessio Corti, and Diletta Martinelli for many useful conversations. This research is supported by EPSRC grant EP/I008128/1.

\bibliographystyle{amsalpha}
\providecommand{\bysame}{\leavevmode\hbox to3em{\hrulefill}\thinspace}
\providecommand{\MR}{\relax\ifhmode\unskip\space\fi MR }
\providecommand{\MRhref}[2]{%
  \href{http://www.ams.org/mathscinet-getitem?mr=#1}{#2}
}
\providecommand{\href}[2]{#2}

\end{document}